\documentclass[english]{amsart}
\usepackage{babel}
\usepackage[]{geometry}
\usepackage{caption}
\usepackage{subcaption}

\usepackage{enumitem}
\usepackage{amsmath,amssymb,amsfonts,latexsym,cancel}
\usepackage{rawfonts}
\usepackage{pictexwd}
\usepackage{tikz}

\usepackage{color}
\usepackage{epsfig}

\newcommand{\R}{\mathbb{R}}

\def\1{\raisebox{2pt}{\rm{$\chi$}}}

\newcommand{\Lip}{\operatorname{Lip}}

\theoremstyle{plain}
\newtheorem{definition}{Definition}[]
\newtheorem{def-thm}{Definition-Theorem}[]

\newtheorem*{proposition ast}{Proposition \ref{Thm X_I charac}*}
\newtheorem{theorem}{Theorem}
\newtheorem{corollary}{Corollary}
\newtheorem{lemma}{Lemma}

\newtheorem{remark}{Remark}

\theoremstyle{definition}

\theoremstyle{remark}

\usepackage{hyperref}

\usepackage[outercaption]{sidecap}

\begin{document}

\title[Reachable set characterization]{Reachable set for Hamilton-Jacobi equations with non-smooth Hamiltonian and scalar conservation laws}

\begin{abstract}
We give a full characterization of the range of the operator which associates, to any initial condition, the viscosity solution at time $T$ of a Hamilton-Jacobi equation with convex Hamiltonian. Our main motivation is to be able to treat the case of convex Hamiltonians with no further regularity assumptions. 
We give special attention to the case $H(p) = |p|$, for which we provide a rather geometrical description of the range of the viscosity operator by means of an interior ball condition on the sublevel sets.
From our characterization of the reachable set, we are able to deduce further results concerning, for instance,  sharp regularity estimates for the reachable functions, as well as structural properties of the reachable set. 
The results are finally adapted to the case of scalar conservation laws in dimension one.
\end{abstract}

\author{Carlos Esteve-Yag\"ue}
\thanks{AMS 2020 MSC: 35F21, 35F25, 49L25 
\textbf{Funding:} 
This project has received funding from the European Research Council (ERC) under the European Union’s Horizon 2020 research and innovation programme (grant agreement NO: 694126-DyCon),  the Alexander von Humboldt-Professorship program, the European Unions Horizon 2020 research and innovation programme under the Marie Sklodowska-Curie grant agreement No.765579-ConFlex, the Transregio 154 Project ‘‘Mathematical Modelling, Simulation and Optimization Using the Example of Gas Networks’’, project C08, of the German DFG, the Grant MTM2017-92996-C2-1-R COSNET of MINECO (Spain) and the Elkartek grant KK-2020/00091 CONVADP of the Basque government. \\
\textbf{Acknowledgements:} 
The authors would like to thank Borjan Geshkovski for his question during a seminar at UAM, which inspired the results presented in this paper.
}
\address {Carlos Esteve-Yag\"ue \newline \indent
{Department of Applied Mathematics and Theoretical Physics, \newline \indent
University of Cambridge,}
\newline \indent
{Wilberforce Road, Cambridge, CB3 0WA, United Kingdom}
}
\email{\texttt{ce423@cam.ac.uk}}

\author{Enrique Zuazua}
\address{Enrique Zuazua \newline \indent
{Chair for Dynamics,Control and Numerics -  Alexander von Humboldt-Professorship
\newline \indent
Department of Data Science, \newline \indent
Friedrich-Alexander-Universit\"at Erlangen-N\"urnberg}
\newline \indent
{91058 Erlangen, Germany}
\newline \indent \hspace{2.5cm} \text{and} \newline \indent
{Chair of Computational Mathematics, Fundación Deusto}
\newline \indent
{Av. de las Universidades, 24}
\newline \indent
{48007 Bilbao, Basque Country, Spain}
\newline \indent \hspace{2.5cm} \text{and} \newline \indent
{Departamento de Matemáticas, \newline \indent
Universidad Autónoma de Madrid,}
\newline \indent
{28049 Madrid, Spain}
}
\email{\texttt{enrique.zuazua@fau.de}}

\date{\today}

\maketitle

\section{Introduction}

We consider first-order Hamilton-Jacobi equations of the form
\begin{equation}
\label{HJ intro}
\left\{
\begin{array}{ll}
\partial_t u + H(\nabla_x u) = 0  & \text{in} \ (0,T) \times \R^N, \\
\noalign{\vspace{2pt}}
u (0, x ) = u_0 (x) & \text{in} \ \R^N,
\end{array}\right.
\end{equation}
where $N\geq 1$, $u_0 \in \Lip(\R^N)$, and the Hamiltonian
$H: \R^N  \longrightarrow \R$ is a given convex function,
with no further regularity assumptions.
It is well-known that the initial-value problem \eqref{HJ intro} is well-posed in the sense of viscosity solutions \cite{crandall1983viscosity,lions1982generalized}.
For any given positive time $T>0$, the main goal of this work is to give a full characterization of the range of the operator 
\begin{equation}\label{forward viscosity operator intro}
\begin{array}{cccc}
S_T^+ : & \Lip(\R^N) & \longrightarrow & \Lip(\R^N) \\
 & u_0 &\longmapsto & u(T,\cdot),
\end{array}
\end{equation}
which associates, to any initial condition, the viscosity solution at time $T$ of the equation \eqref{HJ intro}.

In what follows, the range of the operator $S_T^+$ will be referred to as \emph{the reachable set}, and will be denoted by
\begin{equation}
\label{reachable set}
\mathcal{R}_T := \{ u_T\in \Lip(\R^N) \, : \ \exists u_0 \in \Lip(\R^N) \ \text{such that} \ S_T^+ u_0 = u_T  \} \subset \Lip (\R^N) .
\end{equation}
The problem of characterizing $\mathcal{R}_T$ can be seen as a controllability problem in which the dynamics are governed by the PDE in \eqref{HJ intro}, and the control is the corresponding initial condition.
The characterization of the reachable set for evolutionary equations such as \eqref{HJ intro} is important when addressing the inverse problem of reconstructing the initial condition from an observation of the solution at some positive time $T>0$.
This inverse problem is well-known to be highly ill-posed due to the lack of regularity of the solutions, which gives raise to the loss of backward uniqueness \cite{colombo2020initial,esteve2020inverse,liard2021initial} (multiple initial conditions result in the same solution after some time).
Moreover, in real-life applications, the measurements of the solution are usually noisy, and it is often the case that no initial condition is compatible with the given observation. Hence, when addressing this inverse-design problem, the first step is to construct a reachable function which is as close as possible to the given noisy observation.
This problem can be formulated as a minimum squares problem problem of the form
$$
\underset{\varphi_T\in \mathcal{R}_T}{ \operatorname{minimize}} \| \varphi_T(\cdot) - u_T(\cdot) \|_{L^2}^2,
$$
and is studied in \cite{esteve2021differentiability} for convex smooth Hamiltonians.
Having a good characterization of $\mathcal{R}_T$ is obviously of great interest in order to determine whether existence and uniqueness of minimizers may hold or not, as well as to design optimization algorithms to find a good approximation of the minimizer $\varphi_T^\ast\in \mathcal{R}_T$.

When $H$ is smooth and uniformly convex, i.e.
\begin{equation}
\label{hyp smooth hamiltonian}
H\in C^2 (\R^N) \quad \text{and} \quad
D^2 H(p) \geq c\, I_N \qquad \text{for some} \  c>0,
\end{equation}
the reachable set $\mathcal{R}_T$ is well-studied, and its characterization can be addressed by utilizing semiconcavity\footnote{We recall that a function $f:\R^N \to \R$ is said to be semiconcave if there exists a constant $c\in \R$ such that the function $x\mapsto f(x) - c|x|^2$ is concave.} estimates.
More precisely, it is well-known that  a necessary condition for $u_T\in \mathcal{R}_T$ is given by the following inequality\footnote{Here, $D^2$ stands for the Hessian matrix operator, and the inequality is understood in the usual partial order of symmetric matrices, i.e. $A\leq B$ if and only if $B-A$ is semidefinite positive.} (see \cite{esteve2020inverse,lions2020new})
\begin{equation}
\label{semiconcavity estimate new lions soug}
D^2 u_T  \leq \dfrac{ (D^2H(\nabla u_T))^{-1}}{T} \qquad \text{in} \quad \R^N,
\end{equation}
which is understood in the sense of viscosity solutions.
Moreover, for the one-dimensional case in space, and for quadratic Hamiltonians in any space dimension, it is proven in  \cite[Theorem 2.2]{esteve2020inverse} that the semiconcavity inequality \eqref{semiconcavity estimate new lions soug} is actually optimal, in the sense that \eqref{semiconcavity estimate new lions soug} is equivalent to $u_T\in \mathcal{R}_T$.

In this work, we aim to give similar results for the case when $H:\R^N\longrightarrow \R$ does not  fulfill the hypotheses \eqref{hyp smooth hamiltonian}, and is merely assumed to be a convex function.
In this general context,  where the Hamiltonian is neither smooth nor strictly convex, the viscosity solutions cannot be ensured to be semiconcave, and the (one-sided) regularizing effect of the equation  \eqref{HJ intro} can no longer be expressed by means of differential inequalities such as \eqref{semiconcavity estimate new lions soug}.
Nonetheless, we are still able to give a full characterization of the reachable set $\mathcal{R}_T$ by introducing a global condition, which is based on a family of test functions constructed by means of the Legendre-Fenchel transform of the Hamiltonian.
As we will see in Theorem \ref{thm: HJ abs value intro}, for the level set equation ($H(p) = |p|$), this reachability condition can still be interpreted as a one-sided regularity condition, or semiconcavity condition, not for the solution itself, but for its level sets (see Remark \ref{rmk: semiconcavity level sets}).

\subsection{Characterization of the reachable set}

Let us state our first result, which gives a full characterization of the reachable set for the equation \eqref{HJ intro} when the Hamiltonian is merely assumed to be a convex function.  This characterization identifies the functions $u_T$ in $\mathcal{R}_T$ with those functions such that, for any $x\in \R^N$, there exists a function of the form
$$
z\longmapsto T\, H^\ast \left( \dfrac{z-x_0}{T}  \right)  + c,
$$
touching $u_T$ from above at $x$, where $H^\ast$ is the Legendre-Fenchel transform of the $H$.
Let us recall that the Legendre-Fenchel transform of $H$ is the function
$H^\ast: \R^N \longrightarrow (-\infty, +\infty]$ defined by
 \begin{equation}\label{legendre transform}
 H^\ast (q) = \sup_{p\in \R^N} \left\{ p\cdot q - H(p)  \right\}, \qquad \forall q\in \R^N.
 \end{equation}
Note that the function $H^\ast$ is  convex and lower semicontiuous since it is the supremum of convex continuous functions.
Note also that $H^\ast(q)$ may take infinite values whenever $H$ is not superlinear. 
Indeed,  this is the case for $H(p) = |p|$, whose Legendre-Fenchel transform satisfies $H^\ast (q) = +\infty$ for any $|q|>1$.
 
\begin{theorem}\label{thm: reachability condition}
Let $H: \R^N\to \R$ be a convex function, $u_T\in \Lip(\R^N)$ and $T>0$. 
Set the family of functions 
$$
\mathcal{F}_T (u_T) := 
\left\{ \varphi: z \mapsto T\,  H^\ast \left( \dfrac{z-x_0}{T} \right) + c \ : \quad
x_0 \in \R^N, \ c\in \R \ \text{s.t.} \ \varphi(z) \geq u_T(z) \ \forall z\in \R^N \right\},
$$
where $H^\ast$ is the Legendre-Fenchel transform of $H$ as defined in \eqref{legendre transform}.

Then $u_T\in \mathcal{R}_T$ if and only if for all  $x\in \R^N$, there exists $\varphi \in \mathcal{F}_T(u_T)$ such that $\varphi(x) = u_T(x)$.
\end{theorem}

This characterization is somehow reminiscent of the definition of viscosity subsolution, and can actually be seen as a weaker notion of semiconcavity.
Interesting cases are the power-like Hamiltonians of the form
\begin{equation}
\label{power-like H}
H(p) = \dfrac{|p|^\alpha}{\alpha}  , \qquad \forall p\in \R^N,
\qquad \text{for some $\alpha\in [1, \infty)$.}
\end{equation}
Note that, except for the quadratic case, $\alpha = 2$, Hamiltonians of the form \eqref{power-like H} do not fulfil the hypotheses \eqref{hyp smooth hamiltonian}. 
If we consider $\alpha >1$, then Theorem \ref{thm: reachability condition} implies that for any $T>0$, $u_T\in \mathcal{R}_T$ if and only if, for any $x\in \R^N$, there exists a function of the form
$$
z\longmapsto T \dfrac{\alpha-1}{\alpha} \left| \dfrac{z-x_0}{T}\right|^{\frac{\alpha}{\alpha-1}} + c
$$
touching $u_T$ from above at $x$. 
From this observation, one can deduce the following regularity estimate for the functions in $\mathcal{R}_T$.
The proof of this corollary is given in subsection \ref{subsec:proof corollaries}.

\begin{corollary}
\label{cor: regularity homogeneous hamiltonian}
Let $H$ be of the form \eqref{power-like H} for $\alpha >1$ and $T>0$.
Then, for any $u_T\in \mathcal{R}_T$,  the superdifferential of $u_T(x)$ is nonempty for all $x\in \R^N$, i.e. for all $x\in \R^N$ we have that
$$
D^+ u_T(x) : = \{ q\in \R^N \ : \  \exists \varphi\in C^1 (\R^N) \ \  u_T - \varphi \leq 0, \ u_T (x) - \varphi(x) = 0 \ \nabla \varphi(x) = q  \}\neq \emptyset.
$$
Moreover, the following inequalities hold true:
\begin{enumerate}
\item If $1<\alpha <2$, then the superdifferential
$$
D^2 u_T (x) \leq \dfrac{\Lip (u_T)^{2-\alpha}}{(\alpha-1)T} I_N \qquad \forall x\in \R^N,
$$
where $\Lip (u_T)$ stands for the Lipschitz constant of $u_T$.
\item If $\alpha >2$, then
$$
D^2 u_T (x) \leq  \dfrac{1}{(\alpha-1)T\delta_{x}^{\alpha-2}} I_N \qquad \forall x\in \R^N \quad \text{s.t.} \quad \delta_{x} := \displaystyle\inf_{q\in D^+ u_T(x)} |q|> 0.
$$
\end{enumerate}
\end{corollary}

\begin{remark}
\begin{enumerate}
\item From the statement (i) in the previous Corollary, we deduce that for $\alpha\in (1,2)$ a necessary condition for $u_T\in \mathcal{R}_T$ is that $u_T$ has to be semiconcave with a constant depending on $T$ and the Lipschitz constant of $u_T$.
\item From the statement (ii) we can only deduce a weaker semiconcavity estimate for the regime $\alpha>2$. More precisely,  a semiconcavity estimate only holds at points $x$ which are not critical points of $u_T$.

\item In addition, we observe that if $x$ is a local maximum of $u_T$, then it holds that $D^2 u_T (x)\leq 0$. Hence,  for the case $\alpha>2$, we can slightly improve the result by saying that if $u_T\in \mathcal{R}_T$, then $u_T$ is semiconcave at all points $x\in \R^N$ except for the critical points which are not local maxima (i.e.  local minima and saddle points).
\end{enumerate}
\end{remark}

Let us now look at the limit case $\alpha = 1$,  i.e. when $H$ is given by
\begin{equation}
\label{absolute vale intro}
H(p) = | p | , \qquad \forall p\in \R^N,
\end{equation}
where $|\cdot|$ denotes the euclidean norm  in $\R^N$.
Note that, in this case, $H$ is neither differentiable nor strictly convex, and this brings us to a quite different situation as compared to the regular strictly convex case $\alpha>1$.
The equation \eqref{HJ intro} with $H$ given by \eqref{absolute vale intro} is also known as the level-set equation \cite{osher2001level,osher1988fronts} and is often used to describe the propagation of fronts, evolving in time, as the level sets of the viscosity solution to \eqref{HJ intro}.

In our following result, we will see that, when $H$ is given by \eqref{absolute vale intro},  the reachable target $\mathcal{R}_T$ can be characterized by means of the following interior ball condition on the sublevel sets of $u_T$.

\begin{definition}
\label{def: interior ball}
Let $\Omega\subset \R^N$ be a closed set. We say that $\Omega$ satisfies the interior ball condition with radius $r>0$ if for all $x\in \Omega$, there exists $y\in \Omega$ such that
$$
\overline{B(y,r)} \subset \Omega \qquad \text{and} \qquad
x\in  \overline{B(y,r)}.
$$
\end{definition}

We can now state the following theorem.

\begin{theorem}
\label{thm: HJ abs value intro}
Let $u_T\in \Lip (\R^N)$, $H(p) = |p|$ and $T>0$.
Then $u_T\in \mathcal{R}_T$ if and only if for all $\alpha \in \R$, the $\alpha$-sublevel set defined as
$$
\Omega_\alpha (u_T) : = \{ x\in \R^N\, ; \quad u_T(x) \leq \alpha \}
$$
satisfies the interior ball condition of Definition \ref{def: interior ball} with radius $r=T$.
\end{theorem}

\begin{remark}
\label{rmk: semiconcavity level sets}
\begin{enumerate}
\item Recall that the convexity (resp. concavity) of a set can be characterized by the non-negativity (resp. non-positivity) of the curvature of its boundary.
Taking this into account,  we see that the interior ball condition of Theorem \ref{thm: HJ abs value intro} implies that the curvature of the boundary of any sub-level set of $u_T$ is bounded from above. 
Hence,
the characterization of the reachable set $\mathcal{R}_T$ given in Theorem \ref{thm: HJ abs value intro}
can be seen as a semiconcavity condition on the sublevel sets of $u_T$. In this case, the regularizing effect of the Hamilton-Jacobi equation is not observed on the solution, but rather on its sub-level sets.

\item We point out that the condition of Theorem \ref{thm: HJ abs value intro} is indeed a one-sided regularity estimate for the boundary of the sub-level sets. 
As a matter of fact, the boundary needs not be smooth in general, and might contain corners, which, in view of the interior ball condition,  will always be pointing towards the interior of the sub-level set.
See Figure \ref{fig:level-sets} for an illustration.
\end{enumerate}
\end{remark}

\begin{SCfigure}
\includegraphics[scale=.2]{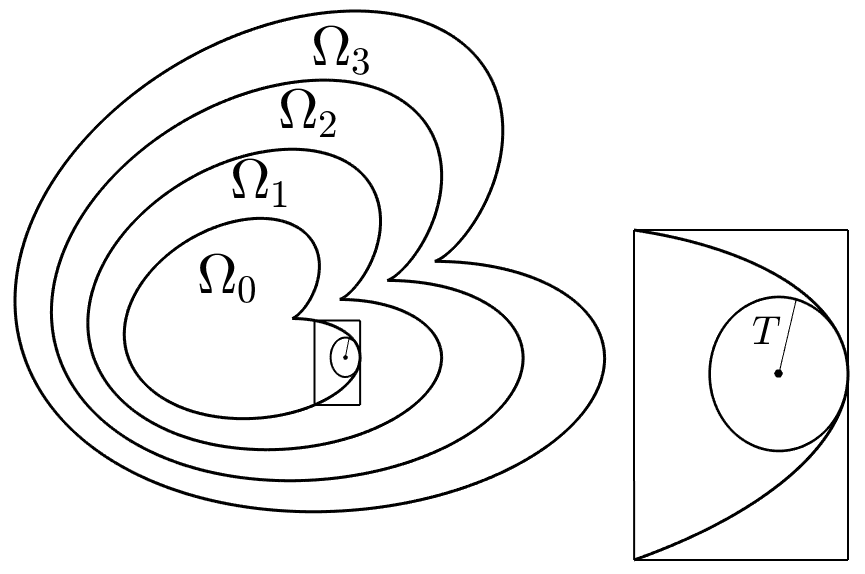}
\caption{The sub-level sets of a function $u_T\in \Lip(\R^2)$ satisfying the interior ball condition from Theorem \ref{thm: HJ abs value intro}.  The region of greatest curvature on the boundary of the sublevel set 0 is zoomed down in the box at the right.}
\label{fig:level-sets}
\end{SCfigure}

In the one-dimensional case in space,  it is sufficient to check the interior ball condition  on the local minima of $u_T$, and then, the above result can be formulated simply as follows:

\begin{corollary}\label{cor:reachability cond example}
Consider the one-dimensional case $N=1$, and let $u_T\in \Lip(\R)$, $H(p) = |p|$ and $T>0$.
Then,  $u_T\in \mathcal{R}_T$ if and only if for any local minimum $x\in \R$  of $u_T$,  there exists $x_0\in \R$ such that $x\in [x_0-T, \, x_0+T]$ and $u_T(y) \leq u_T(x)$ for all $y\in [x_0-T, \, x_0+T]$.
\end{corollary}

See Figure \ref{Fig:example corollary absolute value} for an illustration of this characterization.

\begin{remark}\label{rmk:counterexaple concave are reachable}
In section \ref{sec: HJ}, we shall prove in Corollary \ref{cor: concave are reachable} that, as a consequence of Theorem \ref{thm: reachability condition}, the concave functions satisfy the property of being reachable for all positive times $T>0$.
However, from Corollary \ref{cor:reachability cond example}, we can deduce that for the Hamiltonian $H(p) = |p|$, the concave functions are not the only ones satisfying this property.
Indeed,  if $u_T:\R\to \R$ is monotonically increasing or decreasing, the reachability condition from Corollary \ref{cor:reachability cond example} is trivially satisfied, and then $u_T\in \mathcal{R}_T$ for all $T>0$.  Hence, monotone functions are in $\mathcal{R}_T$ for all $T$ no matter they are concave or not.
We recall that in the smooth strictly convex case, it follows from the necessary condition \eqref{semiconcavity estimate new lions soug},  that $u_T\in \mathcal{R}_T$ for all $T>0$ if and only if $u_T$ is concave.
\end{remark}

\begin{center}
\begin{figure}
\includegraphics[scale=.2]{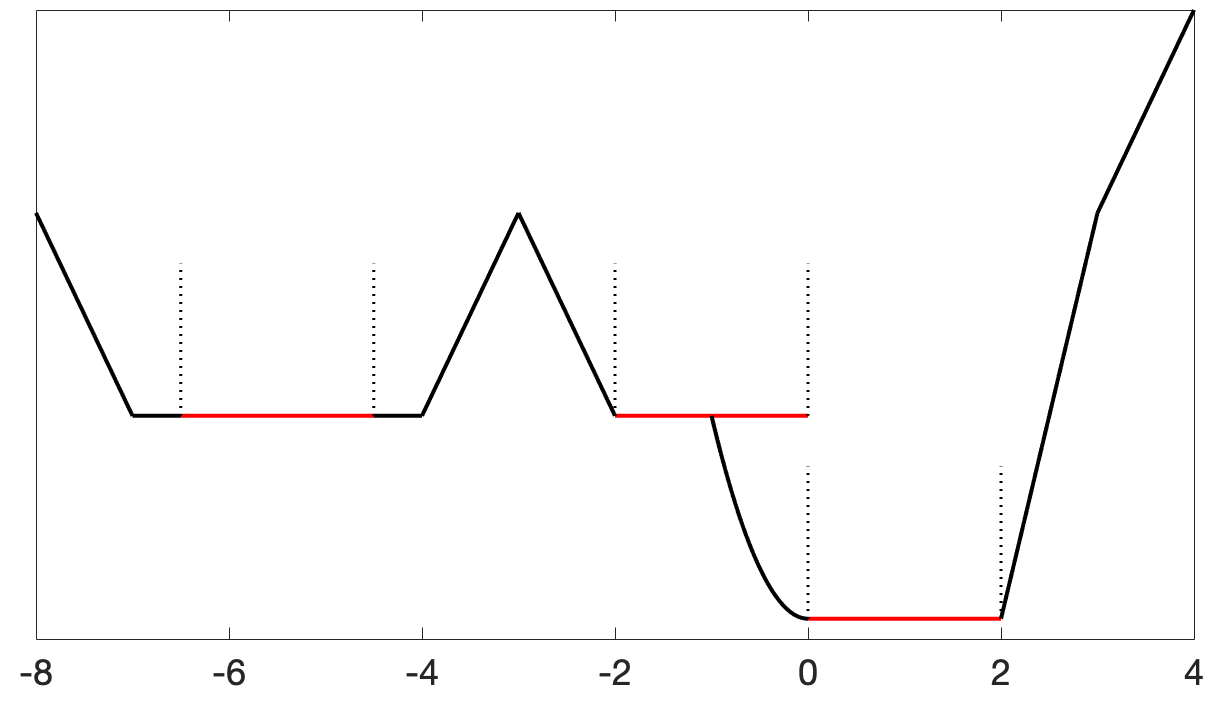}
\caption{An example of function satisfying $u_T\in \mathcal{R}_T$ for the Hamiltonian $H(p) = |p|$ and $T=1$. Observer that, in view of Corollary \ref{cor:reachability cond example}, we only need to check condition (ii) on the local minima of $u_T$.}
\label{Fig:example corollary absolute value}
\end{figure}
\end{center}

\subsection{Structural properties of the reachable set}

As a by-product of the characterization of $\mathcal{R}_T$ given in Theorem \ref{thm: reachability condition}, we can also prove some results concerning the structural properties of the set of reachable functions $\mathcal{R}_T$ for the Hamilton-Jacobi equation \eqref{HJ intro}.  The precise statements of these results are given in subsection \ref{subsec: results HJ}, and their proofs in subsection \eqref{subsec:proof corollaries}.

\begin{enumerate}
\item The reachable set is decreasing in time, i.e. $\mathcal{R}_T\subset \mathcal{R}_{T'}$ for all $0<T'<T$,  and concave functions are reachable for all $T>0$. See Corollary \ref{cor: concave are reachable}.
\item The minimum of two reachable functions is reachable. See Corollary \ref{cor:min of reachables is reachable}.
\item If $H(p) = |p|^\alpha$, with $\alpha \geq 1$, then the reachable set $\mathcal{R}_T$ is star-shaped with center at the origin. See Corollary \ref{cor: star-shaped}.
\item If $H(p) = |p|^2$, then $\mathcal{R}_T$ is convex, and if $H(p) = |p|$, then $\mathcal{R}_T$ is a non-convex cone with vertex at the origin. 
See Corollary \ref{cor: star-shaped}.
\end{enumerate}

\subsection{Reachable set for scalar conservation laws}

In the one-dimensional case in space, it is well-known that Hamilton-Jacobi equations and scalar conservation laws of the form
\begin{equation}
\label{burgers intro}
\partial_t v+ \partial_x[H(v)] = 0 \qquad \text{in} \ (0,T) \times \R
\end{equation}
 are intimately related.
Indeed, if $u\in \Lip( [0,T]\times \R)$ is a viscosity solution to \eqref{HJ intro} with initial condition $u_0\in\Lip(\R)$, then the function $v\in L^\infty((0,T)\times \R)$ given by
$$
v(t,x) = \partial_x u(t,x) \qquad \text{for a.e.} \ (t,x)\in [0,T]\times\R
$$
is the unique entropy solution to \eqref{burgers intro} with initial condition $v_0 = \partial_x u_0$ (see for instance  \cite[Theorem 1.1]{karlsen2000note} and also \cite{caselles1992scalar,colombo2020initial}).

In this section, we adapt the previous results to give a full characterization of the range of the operator
\begin{equation}\label{forward viscosity operator intro SCL}
\begin{array}{cccc}
S_T^{SCL} : & L^\infty (\R) & \longrightarrow & L^\infty (\R) \\
 & v_0 &\longmapsto & v(T,\cdot),
\end{array}
\end{equation}
which associates, to any initial condition $v_0$, the unique entropy solution \cite{dafermos2005hyperbolic, kruvzkov1970first, serre1999systems} to the equation \eqref{burgers intro} at time $T$.
We also define, for any $T>0$, the reachable set for \eqref{burgers intro} as
\begin{equation}
\label{reachable set SCL}
\mathcal{R}_T^{SCL}:= \{ v_T\in L^\infty(\R) \, : \ \exists v_0 \in L^\infty(\R^N) \ \text{such that} \ S_T^{SCL} v_0 = v_T  \} \subset L^\infty (\R),
\end{equation}

For the scalar conservation law \eqref{burgers intro} with a flux $H$ satisfying \eqref{hyp smooth hamiltonian}, it is well-known \cite{colombo2020initial,dafermos1977generalized, escobedo1993asymptotic, hoff1983sharp} that for any $T>0$ and $v_T\in L^\infty(\R)$, the property $v_T\in \mathcal{R}_T^{SCL}$ is equivalent to the one-sided-Lipschitz condition
\begin{equation}
\label{OSL SCL improved intro}
\partial_{p}H(v (t,y)) - \partial_p H (v (t,x)) \leq \dfrac{y-x}{t} \qquad \text{for a.e.} \ x\leq y.
\end{equation}

In the general convex case, in which $H$ is not necessarily differentiable nor strictly convex,  the one-sided-Lipschitz inequality \eqref{OSL SCL improved intro} does not hold in general. Nonetheless, we can adapt Theorem \ref{thm: reachability condition} in the following way to give a full characterization of the functions in $\mathcal{R}_T^{SCL}$.

\begin{theorem}
\label{thm: reach charac SCL}
Let $H: \R \to \R$ be a convex function,  $v_T\in L^\infty(\R)$ and $T>0$.
Then $v_T\in \mathcal{R}_T^{SCL}$   if and only if
\begin{equation}\label{SCL cond}
\begin{array}{c}
\text{for all} \  x\in \R, \ \text{there exists} \ x_0 \in \R \ \text{such that the function} \\
z \longmapsto\displaystyle\int_0^z v_T (y) dy - T \, H^\ast \left( \dfrac{z-x_0}{T}\right) \ \text{has a global maximum at $x$}.
\end{array}
\end{equation}
\end{theorem}

Sharp one-sided regularity estimates are for power-like fluxes of the form $|p|^\alpha$ with $\alpha>1$ are given in \cite{escobedo1993asymptotic}. 
The limit case $\alpha = 1$ is again different since $H$ is no longer differentiable.
The following theorem provides a full characterization of the functions in $\mathcal{R}_T^{SCL}$, when the flux is the absolute value.

\begin{theorem}
\label{thm: reachability SCL abs val}
Let $v_T\in L^\infty(\R)$,  $H(p) = |p|$ and $T>0$. Then,   $v_T\in \mathcal{R}_T^{SCL}$ if and only if
\begin{equation}
\label{OSLC abs value}
\operatorname{sgn} (v_T(y)) - \operatorname{sgn} (v_T(x)) \leq \dfrac{y-x}{T} \qquad \text{for a.e.} \ x,y\in \operatorname{supp} (v_T) \ \text{satisfying} \ x\leq y.
\end{equation}
\end{theorem}

Here, the sign function $\operatorname{sgn}: \R\setminus \{0\} \to \{-1, 1\}$ is defined as
$$
\operatorname{sgn} (z) :=
\left\{
\begin{array}{ll}
-1 & \text{if} \ z<0 \\
1 & \text{if} \ z>0.
\end{array}
\right. 
$$

The above result must be interpreted as follows:
in order for $v_T$ to be reachable, any sign change, from negative to positive, must be separated by an interval of length $2T$ where $v_T$ vanishes.
More precisely, if we define the supports of the positive and negative parts of $v_T$
$$
A_+ = \{ x \in \R \, ; \ v_T(x)>0\} \qquad \text{and} \qquad
A_- = \{ x\in \R \, ; \ v_T(x) <0 \}.
$$
then it must hold that
$$
y-x\geq 2T, \qquad \text{for a.e.} \  x\in A_- \ \text{and for a.e. } \ y\in A_+ \ \text{with} \ x\leq y.
$$
See Figure \ref{Fig:example Burgers  absolute value} for an illustration of a function $v_T$ satisfying this property.

\begin{center}
\begin{figure}
\includegraphics[scale=.2]{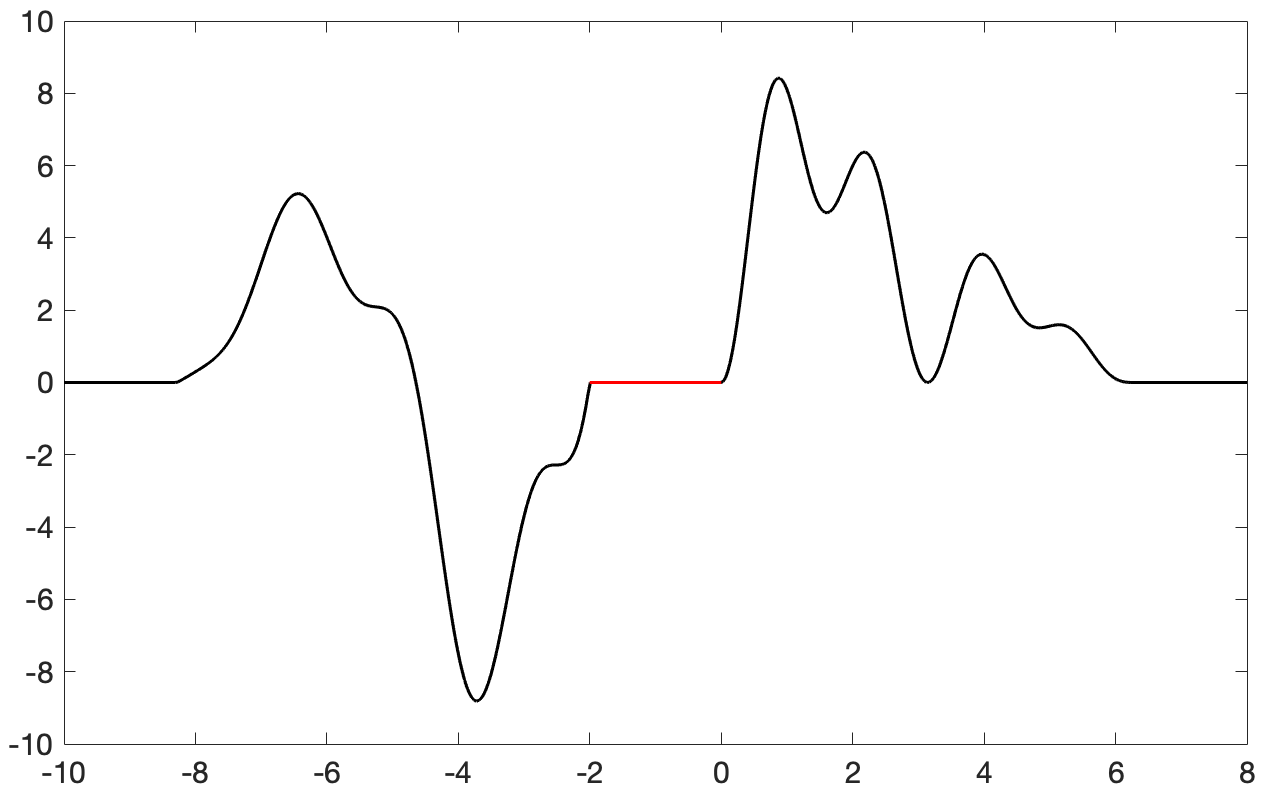}
\caption{An example of a function $v_T\in \mathcal{R}_T^{SCL}$ for the flux $H(p) = |p|$ and $T=1$. }
\label{Fig:example Burgers  absolute value}
\end{figure}
\end{center}

The rest of the paper is structured as follows.
Section \ref{sec: HJ} is devoted to Hamilton-Jacobi equations. In subsection \ref{subsec: results HJ}, we present some corollaries concerning the structural properties of $\mathcal{R}_T$, that can be deduced from Theorem \ref{thm: reachability condition}.   Then, in subsection \ref{subsec: preliminaries} we give some prelimiaries about Hamilton-Jacobi equations and the Hopf-Lax formula which are then used in subsections \ref{subsec: proof reachability} and \ref{subsec:proof corollaries} to prove our results.
In section \ref{sec: SCL}, we prove the characterization of the reachable set given in Theorem \ref{thm: reach charac SCL} for scalar conservation laws \eqref{burgers intro} with general convex flux,  and we also prove Theorem \ref{thm: reachability SCL abs val} for the case when the flux is the absolute value.
Finally, we conclude the paper with a section describing our conclusions and presenting a couple of open questions.

\section{Hamilton-Jacobi equations}
\label{sec: HJ}

In this section, we deal with Hamilton-Jacobi equations of the form \eqref{HJ intro} with a convex Hamiltonian $H: \R^N\to \R$,  and without making any further regularity assumptions.
As announced in the introduction, for a given $T>0$, our main goal is to prove the full characterization (necessary and sufficient condition) given in Theorem \ref{thm: reachability condition} for the reachable set $\mathcal{R}_T$, defined as in \eqref{reachable set}, and also prove its main properties.
Before addressing the proofs of our results, let us state in the following subsection the results concerning the structural properties of $\mathcal{R}_T$, that can be deduced from Theorem \ref{thm: reachability condition}.

\subsection{Reachable set: main properties}
\label{subsec: results HJ}

Theorem \ref{thm: reachability condition} has some interesting consequences, revealing information about the structure of the reachable set $\mathcal{R}_T$, and the way it evolves as we increase the time horizon $T$.
The following result ensures that the reachable set decreases as $T$ increases, and that concave functions have the property of being reachable for all $T>0$.
The corollary is proved in subsection \ref{subsec:proof corollaries}.
 
 \begin{corollary}\label{cor: concave are reachable}
 Let $H:\R^N\to \R$ be a convex function. Then,
 $$
 \text{for any} \quad 0<T'<T, \quad
 \text{we have} \quad \mathcal{R}_T \subset \mathcal{R}_{T'}.
 $$
Moreover, if $u_T\in \Lip(\R^N)$ is a concave function, then,  $u_T \in \mathcal{R}_T$ for all $T>0$.
 \end{corollary}
 
 \begin{remark}\label{rmk: concave targets}
Corollary \ref{cor: concave are reachable} states that concavity is a sufficient condition for a function to be reachable for all $T>0$.  However, it is not necessary in general. Indeed,  it can be proved (see Remark \ref{rmk:counterexaple concave are reachable}) that, if one considers the one-dimensional case with the Hamiltonian given by $H(p) = |p|$,  any globally Lipschitz monotone (increasing or decreasing) function $u_T:\R\to \R$ is reachable for all $T>0$, even if it is not concave.
 It differs from the smooth uniformly convex case \eqref{hyp smooth hamiltonian}, where,  due to the necessary condition \eqref{semiconcavity estimate new lions soug}, a function is reachable for all $T>0$ if and only if it is a concave function.
 \end{remark}
 
Another interesting consequence of Theorem \ref{thm: reachability condition} is the following corollary, which roughly says that the minimum of two reachable targets is reachable.
The proof of this corollary is omitted as it is a straightforward consequence of Theorem \ref{thm: reachability condition}.

 \begin{corollary}\label{cor:min of reachables is reachable} 
 Let $T>0$, let $H:\R^N\to \R$ be a convex function, and let $H^\ast$ be its Legendre-Fenchel transform as defined in \eqref{legendre transform}.  Then, the following statements hold true.
\begin{enumerate}
\item For any $u_T,v_T\in \mathcal{R}_T$, the function
$w_T(x) := \min \{ u_T(x), v_T(x) \}$
satisfies $w_T\in \mathcal{R}_T$.

\item If in addition $H^\ast$ is locally Lipschitz,  then for any $u_T\in \mathcal{R}_T$, $x_0\in \R^N$ and $c\in \R$, the function
$$
w_T(x) := \min \left\{  u_T(x), \ T\, H^\ast\left( \dfrac{x-x_0}{T} \right) + c \right\}
$$
satisfies $w_T\in \mathcal{R}_T$.
\end{enumerate} 
 \end{corollary}
 
 Note that in (ii), the assumption of $H^\ast$ being a locally Lipschitz continuous function is needed to guarantee that $w_T\in \Lip(\R^N)$. 
 Corollary \ref{cor:min of reachables is reachable} provides, in particular, a simple method to construct reachable functions with compact support when $H^\ast$ is locally Lipschitz.
Note that the zero function is reachable for any $T>0$. Then, for any given finite set $\{ (x_i,c_i) \}_{i=1}^k\subset \R^N\times \R$, we can define the function
$$
u_T (x) := \min\left\{ 0, \ T\, H^\ast \left( \dfrac{x-x_1}{T} \right) + c_1, \ \ldots , \  T\, H^\ast \left( \dfrac{x-x_k}{T} \right) + c_k \right\}
$$
which, in view of Corollary \ref{cor:min of reachables is reachable}, satisfies $u_T\in \mathcal{R}_T$.
Of course, the method can readily be applied to larger collections of points $\{ (x_i,c_i) \}_\mathcal{I} \subset \R^N\times \R$, under the assumption of $c_i$ being uniformly bounded from below.
See Figure \ref{Fig:example reachable targets} for an illustration of this result.

\begin{center}
\begin{figure}
\includegraphics[scale=.18]{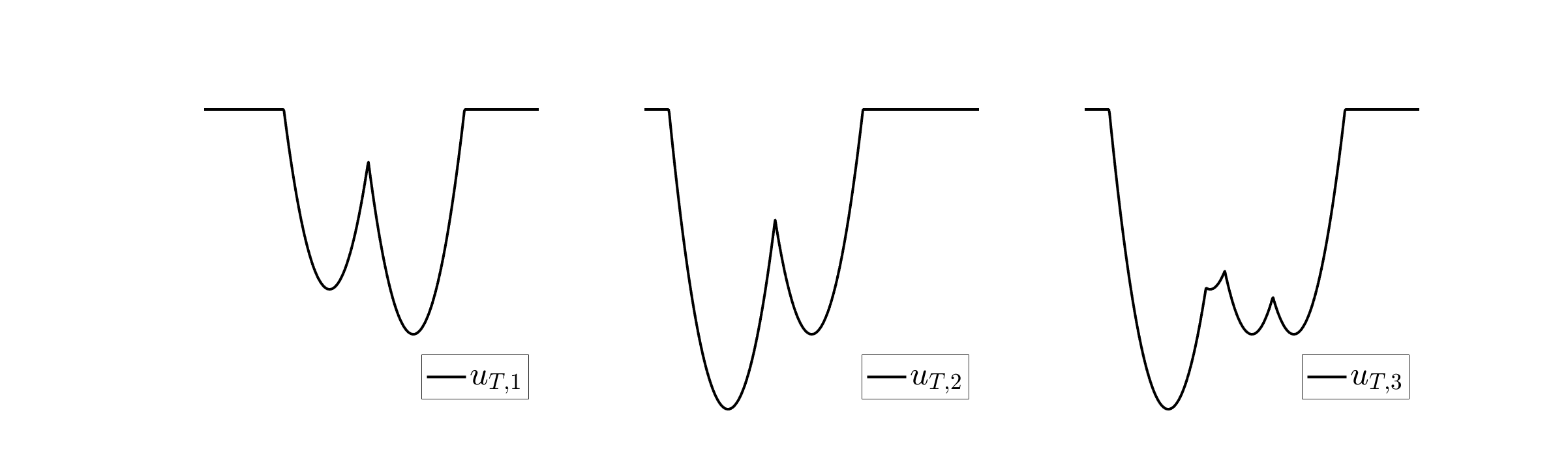}
\caption{Three examples of reachable targets in time $T=1$, with compact support,  for the Hamiltonian $H(p) = |p|^2/2$. The first two examples were constructed using the statement (ii) in Corollary \ref{cor:min of reachables is reachable}, whereas the third one was constructed using the statement (i) as the minimum of the two first examples.}
\label{Fig:example reachable targets}
\end{figure}
\end{center}

The last property about the reachable set $\mathcal{R}_T$ that we are going to present as a consequence of Theorem \ref{thm: reachability condition} applies to power-like Hamiltonians of the form \eqref{power-like H}.
The following corollary ensures that the reachable set $\mathcal{R}_T$ is \emph{star-shaped} with center the origin, i.e. 
\begin{equation}
\label{star-shaped}
\forall u_T\in \mathcal{R}_T\quad \text{and} \quad  \forall \lambda\in [0,1],
\quad \text{we have} \quad
 \lambda u_T\in \mathcal{R}_T.
\end{equation}
For the particular case $\alpha = 2$,  the set $\mathcal{R}_T$ is additionally convex, and if $\alpha = 1$, then $\mathcal{R}_T$ is actually a non-convex cone with vertex at the origin, i.e.
\begin{equation}
\label{cone def}
\forall u_T\in \mathcal{R}_T\quad \text{and} \quad  \forall \lambda\in [0,\infty),
\quad \text{we have} \quad
 \lambda u_T\in \mathcal{R}_T.
\end{equation}

\begin{corollary}
\label{cor: star-shaped}
Let $T>0$ and let $H:\R^N\to \R$ be given by \eqref{power-like H} for some $\alpha\in [1, \infty)$.
Then  the reachable set $\mathcal{R}_T$ is star-shaped with center the origin, i.e. \eqref{star-shaped} holds.
 Moreover, if $\alpha = 2$,  then $\mathcal{R}_T$ is convex, and
if $\alpha =1$, then $\mathcal{R}_T$ is a cone with vertex at the origin.
\end{corollary}

The proof of the corollary is given in subsection \ref{subsec:proof corollaries}.

\subsection{Preliminaries}
\label{subsec: preliminaries}

Let us recall some elementary facts about  viscosity solutions to Hamilton-Jacobi equations of the form \eqref{HJ intro} that are well-known in the literature and will be used throughout our proofs.
Let us recall from \eqref{forward viscosity operator intro} in the introduction that, for any $T>0$, the (forward) viscosity operator $S_T^+$ associates, to any initial condition $u_0$, the viscosity solution to \eqref{HJ intro} at time $t=T$.
It is well-known that the viscosity solution to \eqref{HJ intro} can be given by the so-called Hopf-Lax formula (see for instance \cite{alvarez1999hopf,bardi1984hopf,barles1994solutions}). Then, for any $T>0$, the operator $S_T^+$ can be explicitly defined as
 \begin{equation}\label{HopfLax formula}
 S_T^+ u_0 (x) = \min_{y\in\R^N} \left\{ u_0(y) + T\, H^\ast \left( \dfrac{x-y}{T} \right)  \right\},
 \end{equation}
 where $H^\ast$ is defined as in \eqref{legendre transform}.

The simplest way to characterize the reachable set $\mathcal{R}_T$, which actually applies to more general Hamiltonians of the form $H(x,p)$,  is to perform a backward-forward resolution of \eqref{HJ intro}, 
 by means of the so-called \emph{backward viscosity operator} (see \cite{barron1999regularity})
 $$
 \begin{array}{cccc}
 S_T^- : & \Lip(\R^N) & \longrightarrow & \Lip(\R^N) \\
 & u_T & \longmapsto & S_T^- u_T = w(0,\cdot),
 \end{array}
 $$
 where $w\in \Lip([0,T]\times \R^N)$ is the unique \emph{backward viscosity solution} to \eqref{HJ intro} with terminal condition $u_T$.
We recall that $w\in \Lip([0,T]\times \R^N)$ is a backward viscosity solution to \eqref{HJ intro} if and only if the function $v(t,c) := w(T-t,x)$ is a forward viscosity solution to
 $$
 \partial_t v - H(D_x v) = 0 \qquad \text{in}\ [0,T]\times\R^N.
 $$
 As well as for the forward viscosity solutions, existence, uniqueness and stability of backward viscosity solutions for the terminal value problem associated to the Hamilton-Jacobi equation \eqref{HJ intro} can be proved by means of the vanishing viscosity method, i.e. the backward viscosity solution can be obtained as the limit when $\varepsilon\to 0^+$ of the solution $u_\varepsilon$ to the terminal value problem
 $$
\left\{\begin{array}{ll}
\partial_t u_\varepsilon + \varepsilon \Delta u_\varepsilon + H (D_x u_\varepsilon) = 0 & \text{in} \ (0,T)\times \R^N \\
u_\varepsilon(T, \cdot)= u_T & \text{in} \ \R^N.
\end{array}\right.
$$
 
  Let us now recall the reachability condition for the initial-value problem \eqref{HJ intro} which, for any $T>0$, identifies the reachable targets in time $T$ with the fix points of the composition operator $S_T^+\circ S_T^-$.
Under the assumption of  $H:\R^N\to \R$ being a convex function and $u_T\in \Lip(\R^N)$, we have that
\begin{equation}\label{reachability condition back-forth}
u_T\in \mathcal{R}_T \quad \text{if and only if} \quad S_T^+\circ S_T^- u_T = u_T
\end{equation}
 The proof of \eqref{reachability condition back-forth} is exactly the same as the one of \cite[Theorem 2.1]{esteve2020inverse},  which is a direct consequence of \cite[Proposition 4.7]{esteve2020inverse} (see also \cite{barron1999regularity,misztela2020initial}), and we omit the proof here.

As well as for the forward viscosity solutions,  there is a Hopf-Lax formula for the backward viscosity solutions to \eqref{HJ intro} with terminal condition $u_T\in\Lip(\R^N)$, which reads as
 \begin{equation}\label{HopfLax backward}
 S_T^- u_T (x) = \max_{y\in\R^N} \left\{ u_T(y) - T\, H^\ast \left( \dfrac{y-x}{T} \right)  \right\}.
 \end{equation}

Let us  finish the subsection with the proof of  the following elementary property of $H^\ast$, which will be used in the sequel. 
 
 \begin{lemma}\label{lem: L coercive}
 Let $H:\R^N \to \R$ be a convex function and let $H^\ast$ be its Legendre-Fenchel transform. Then,  for any constant $C>0$,  we have
 $$
 H^\ast(q) \geq C\, |q| - \max_{p\in\overline{B(0,C)}} H(p) \qquad \forall q\in \R^N,
 $$
 where $\overline{B(0,C)}$ is the closure of the ball of radius $C$ centered at the origin.
 \end{lemma}

\begin{proof}
Let $C>0$ be any positive constant.  Since $H$ is convex and takes values in $\R$, we deduce that $H$ is continuous, and then we have
$$
\max_{p\in \overline{B(0,C)}} H(p) <\infty.
$$
Now, using the definition of  $H^\ast$ in \eqref{legendre transform}, for any $q\in \R^N$, we can take $p=C\frac{q}{|q|}\in \overline{B(0,C)}$ and then deduce that
$$
H^\ast(q) \geq C|q| - H \left(  C\dfrac{q}{|q|} \right) \geq 
C|q| - \max_{p\in \overline{B(0,C)}} H(p).
$$
\end{proof}

 \subsection{Proof of Theorems \ref{thm: reachability condition} and  \ref{thm: HJ abs value intro} }
\label{subsec: proof reachability}

We start with the proof of Theorem \ref{thm: reachability condition}.
 
\begin{proof}[Proof of Theorem \ref{thm: reachability condition}]
Let $u_T\in \mathcal{R}_T$ be a reachable target. By \eqref{reachability condition back-forth}, we have that, for all $x\in \R^N$, there exists $x_0\in \R^N$ such that
$$
u_T (x) = S_T^- u_T(x_0) + T\, H^\ast \left( \dfrac{x-x_0}{T} \right).
$$
Using the definition of $S_T^-$ in \eqref{HopfLax backward}, we deduce that 
$$
S_T^- u_T(x_0) \geq u_T (z) - T\, H^\ast \left( \dfrac{z-x_0}{T} \right)
\qquad \forall z\in \R^N.
$$
Hence, by setting $c= S_T^- u_T(x_0)$, we obtain that the function
$$
\varphi  (z) = c + T\, H^\ast \left( \dfrac{z-x_0}{T} \right)
$$  
satisfies $\varphi\in \mathcal{F}_T(u_T)$ and $\varphi(x) = u_T(x)$.

For the reverse implication, let us first prove that,  for any $u_T\in \Lip(\R^N)$, it holds that
\begin{equation}\label{ineq reach cond proof}
u_T (x) \leq  S_T^+\circ S_T^- u_T(x) \qquad \forall x\in \R^N.
\end{equation}
In view of \eqref{HopfLax backward}, we have
$$
S_T^- u_T(y) \geq u_T(x) - T\, H^\ast \left( \dfrac{x-y}{T} \right) \qquad \forall x,y\in \R^N,
$$
which implies that
$$
u_T(x) \leq \min_{y\in \R^N} \left\{ S_T^- u_T(y) + T\, H^\ast \left( \dfrac{x-y}{T} \right)\right\} = S_T^+\circ S_T^- u_T(x) \qquad \forall x\in\R^N. 
$$

Now, let $u_T\in \Lip(\R^N)$ be  such that, for all $x\in \R^N$, there exists $\varphi\in \mathcal{F}_T(u_T)$ satisfying $\varphi(x) = u_T(x)$.  This means that there exist $x_0\in \R^N$ and $c\in \R$ such that
$$
u_T (x) = c + T\, H^\ast\left( \dfrac{x-x_0}{T} \right)
$$
and 
$$
u_T (z) \leq c + T\, H^\ast\left( \dfrac{z-x_0}{T} \right) \qquad \forall z\in \R^N.
$$
This  in particular implies, as a consequence of \eqref{HopfLax backward}, that $c = S_T^-u_T(x_0)$.
Hence,  using \eqref{HopfLax formula} we deduce that
\begin{eqnarray*}
S_T^+ \circ S_T^- u_T(x) & = & \min_{y\in \R^N} \left\{ S_T^- u_T(y) + T\, H^\ast\left(\dfrac{y-x}{T}  \right) \right\} \\
&\leq &  S_T^- u_T(x_0) + T\, H^\ast\left(\dfrac{x_0-x}{T}  \right)  =u_T(x).
\end{eqnarray*}
Combining this inequality with \eqref{ineq reach cond proof} we deduce that $S_T^+\circ S_T^- u_T (x)= u_T(x)$ for all $x\in \R^N$, and then we can use the general reachability criterion  \eqref{reachability condition back-forth} to deduce that $u_T\in \mathcal{R}_T$.
\end{proof}

Let us now prove Theorem \ref{thm: HJ abs value intro} using the conclusion of Theorem \ref{thm: reachability condition}.

\begin{proof}[Proof of Theorem \ref{thm: HJ abs value intro}]
Note first of all that the Legendre-Fenchel transform of $H(p) = |p|$ is given by
\begin{equation}
\label{Lengendre transform abs value}
H^\ast (q) = \left\{
\begin{array}{ll}
0 & \text{if} \ |q| \leq 1 \\
\noalign{\vspace{2pt}}
+\infty & \text{if}\ |q|>1.
\end{array}
\right.
\end{equation}

In view of the form of $H^\ast$, the functions in $\mathcal{F}_T$ defined in the statement of Theorem \ref{thm: reachability condition} are simply functions which are constant in a ball of radius $T$ and infinity elsewhere.
Therefore, the reachability condition from Theorem \ref{thm: reachability condition}, in this case, reads as follows:
\begin{equation}
\label{interior ball proof}
\forall x\in \R^N, \quad \exists x_0 \in \R^N \quad \text{such that} \quad x\in \overline{B(x_0,T)} \quad \text{and} \quad
u_T (y)\leq u_T(x) \quad
 \forall y\in  \overline{B(x_0,T)}.
\end{equation}

It is easy to prove that  this property is equivalent to the interior ball condition from Definition \ref{def: interior ball} with $r=T$.
Let us first assume that \eqref{interior ball proof} holds. Then, for any $\alpha\in \R$ and $x\in \Omega_\alpha (u_T)$, we have that there exists a ball $\overline{B(x_0,T)}$ containing $x$ such that
$$
u_T (y)\leq u_T(x) \leq \alpha, \qquad \forall y\in \overline{B(x_0,T)},
$$
which implies that $\overline{B(x_0,T)}\subset \Omega_\alpha (u_T)$.

On the other hand, if the interior ball condition holds with $r=T$, then for any $x\in \R^N$ we have that $x \in  \Omega_\alpha (u_T)$ with $\alpha = u_T(x)$. Hence, by the interior ball condition, there exists $x_0\in \Omega_\alpha(u_T)$ such that $x\in \overline{B(x_0, T)} \subset \Omega_\alpha (u_T)$, which then implies that
$$
u_T(y) \leq \alpha = u_T(x) \qquad \forall y\in \overline{B(x_0, T)}.
$$
\end{proof}

\subsection{Proof of Corollaries \ref{cor: regularity homogeneous hamiltonian},  \ref{cor: concave are reachable} and \ref{cor: star-shaped}}
\label{subsec:proof corollaries}

We start by proving the regularity result given in Corollary \ref{cor: regularity homogeneous hamiltonian}  for power-like Hamiltonians.

\begin{proof}[Proof of Corollary \ref{cor: regularity homogeneous hamiltonian}]
We start by noticing that, since $H(p) = |p|^\alpha$ for some $\alpha >1$, the its Legendre--Fenchel transform is given by
$$
H^\ast (q) = \frac{\alpha - 1}{\alpha} | q |^{\frac{\alpha}{\alpha-1}}.
$$
Then, a straightforward computation gives the following:
\begin{equation}
\label{grad Hstar}
\nabla H^\ast (q) = |q|^{\frac{2-\alpha}{\alpha -1}} q \qquad \forall q\in \R^N,
\end{equation}
and
\begin{equation}
\label{hessian Hstar estimate}
D^2 H^\ast (q) \leq \frac{1}{\alpha-1} |q |^{\frac{2-\alpha}{\alpha-1}} I_N, \qquad \forall q\in \R^N.
\end{equation}

Now,  from Theorem \ref{thm: reachability condition}, we have that if $u_T\in \mathcal{R}_T$, then for any $x\in \R^N$ there exists a function $\varphi:\R^N\to \R$ of the form
$$
\varphi(z) : = u_T (z) - T \, H^\ast \left( \dfrac{z-x_0}{T}\right) -c
$$
for some $x_0\in \R^N$ and $c\in \R$ such that $\varphi(\cdot)$ attains its maximum at $x$.

This implies that $0\in D^+ \varphi (x)$, which then implies, using \eqref{grad Hstar}, that
\begin{equation}
\label{grad Hstar in subdiff}
\nabla H^\ast \left( \dfrac{x-x_0}{T} \right) = \left| \dfrac{x-x_0}{T}  \right|^{\frac{2-\alpha}{\alpha-1}} \dfrac{x-x_0}{T} \in D^+ u_T(x).
\end{equation}
It then follows that $D^+ u_T(x)\neq \emptyset$ for all $x\in \R^N$.

Let us now prove the semiconcavity inequalities.
Since $\varphi(\cdot)$ attains its maximum at $x$, we have that its Hessian matrix at $x$ is semidefinite negative, i.e. $D^2 \varphi (x)$.
Then, using \eqref{hessian Hstar estimate} we obtain that
\begin{equation}
\label{semiconcav est proof}
\begin{array}{rcl}
D^2 u_T (x) & \leq & \dfrac{1}{T} D^2 H^\ast \left( \dfrac{x-x_0}{T} \right) \\
\noalign{\vspace{2pt}}
& \leq & \dfrac{1}{\alpha-1} \dfrac{|x-x_0|^{\frac{2-\alpha}{\alpha-1}}}{T^{\frac{1}{\alpha-1}}} I_N.
\end{array}
\end{equation}

We now need to use an estimate for the quantity $|x-x_0|$, taking into account that the exponent $\frac{2-\alpha}{\alpha-1}$ has different sign depending whether $\alpha\in (1,2)$ or $\alpha>2$.

If $\alpha\in (1,2)$, we can use \eqref{grad Hstar in subdiff}, and the Lipschitz constant of $u_T$, that we denote by $\Lip(u_T)$, to deduce that
$$
|x-x_0| \leq T \Lip(u_T)^{\alpha-1}.
$$
Note that $u_T\in \Lip(\R^N)$ implies that $|q|\leq \Lip(u_T)$ for all $q\in D^+ u_T(x)$, and for all $x\in \R^N$.

Hence, combining the above inequality with \eqref{semiconcav est proof},  along with the fact that the exponent $\frac{2-\alpha}{\alpha-1}$ is positive, we deduce that
$$
D^2 u_T (x) \leq \dfrac{\Lip(u_T)^{2-\alpha}}{(\alpha -1)T} I_N.
$$

Let us now assume that $\alpha>2$. If we have
$$
\delta_x : = \inf_{q\in D^+ u_T(x)} |q| >0,
$$
we can deduce from \eqref{grad Hstar in subdiff} that
$$
|x-x_0| \geq T\, \delta_x^{\alpha-1}.
$$
Hence,  combining this with \eqref{semiconcav est proof}, and the fact that the exponent $\frac{2-\alpha}{\alpha-1}$ is negative, we deduce that
$$
D^2 u_T(x) \leq \dfrac{1}{(\alpha -1)\delta_x^{\alpha -2} T} I_N.
$$
\end{proof}

Let us now prove Corollary \ref{cor: concave are reachable}.
 
  \begin{proof}[Proof of Corollary \ref{cor: concave are reachable}]
The fact that the reachable set $\mathcal{R}_T$ decreases in time is a direct consequence of the semigroup property of $S_T^+$. Indeed, if $u_T\in \mathcal{R}_T$, then there exists $u_0\in \Lip(\R^N)$ such that $S_T^+ u_0 = u_T$.
Then, for any $T'\in (0,T)$,  consider the initial condition
$$
\tilde{u}_0(x) = S_{T-T'}^{HJ} u_0(x).
$$
By the semigroup property we have that 
$$
S_{T'}^{HJ} \tilde{u}_0 (x) = S_{T'}^{HJ} (S_{T-T'}^{HJ} u_0(x)) = S_T^+ u_0 (x) = u_T(x),
$$
 implying that $u_T\in \mathcal{R}_{T'}$.
This proves that $\mathcal{R}_T \subset \mathcal{R}_{T'}$.
  
Let us now prove the second part of the Corollary.
Let $u_T\in \Lip(\R^N)$ be concave and fix any $T>0$.
 In view of Theorem \ref{thm: reachability condition}, it suffices to prove that, for all $x\in \R^N$, there exists $x_0\in \R^N$ and $c\in \R$ such that 
 \begin{equation}\label{reach cor proof}
 u_T(x) = T\, H^\ast\left(  \dfrac{x-x_0}{T}\right) + c \quad \text{and} \quad
 u_T(z) \leq  T\, H^\ast\left(  \dfrac{z-x_0}{T}\right) + c \quad \forall z\in \R^N.
 \end{equation}
 
 Since $u_T$ is concave, for any $x\in \R^N$, there exists $p_0\in \R^N$  such that 
 \begin{equation}\label{ineq cor aux 1}
 u_T(z) \leq p_0\cdot (z-x) + u_T(x) \qquad \forall z\in \R^N.  
 \end{equation}
 
On the other hand, it is well-known that the convex conjugate of a convex function is convex and lower semi-continuous.
This, combined with the superlinearity of $H^\ast$ proved in Lemma \ref{lem: L coercive}, implies the existence of $q_0\in \R^N$ satisfying
$$
H^\ast(q_0) - p_0\cdot q_0 = \min_{q\in \R^N} \left\{ H^\ast(q) - p_0\cdot q  \right\} > -\infty,
$$
and then we have
\begin{equation}\label{ineq cor aux 2}
H^\ast(q) \geq p_0 \cdot (q-q_0) + H^\ast(q_0) \qquad \forall q\in \R^N.
\end{equation}

Set $x_0 := x-Tq_0$. For any $z\in \R^N$, we can plug $q= \frac{z-x_0}{T}$ into \eqref{ineq cor aux 2} and multiply by $T$ to obtain
$$
T\, H^\ast\left( \dfrac{z-x_0}{T}  \right) \geq p_0 \cdot (z-x) + T\, H^\ast(q_0).
$$

Finally, combining this inequality with \eqref{ineq cor aux 1}, we deduce that
$$
u_T(z) \leq T\, H^\ast \left( \dfrac{z-x_0}{T} \right) - T\, H^\ast(q_0) + u_T(x), \quad  \forall z\in \R^N.
$$
By the choice of $x_0$, we observe that the above inequality is actually an equality for $z =x$. 
Then \eqref{reach cor proof} follows with $c = u_T(x) - T\, H^\ast(q_0)$, and the corollary is proved.
 \end{proof}

We end the section with the proof of Corollary \ref{cor: star-shaped}.

\begin{proof}[Proof of Corollary \ref{cor: star-shaped}]
We start with the cases $\alpha = 2$ and $\alpha =1$.
The case $\alpha = 2$ follows directly from the characterization of the $\mathcal{R}_T$ given by the semiconcavity condition \eqref{semiconcavity estimate new lions soug}, which in this case reads as
$$
D^2 u_T  \leq \dfrac{I_N}{T}, \qquad \text{in the viscosity sense.}
$$
Note that if $u_T$ and $v_T$ both satisfy this inequality, then so does $\lambda u_T + (1-\lambda)v_T$ for all $\lambda\in [0,1]$.

The case $\alpha = 1$ follows from Theorem \ref{thm: HJ abs value intro}.
Indeed, for any $u_T\in \mathcal{R}_T$ and $\lambda>0$ we have that, for all $\gamma\in\R$ the sublevel set $\Omega_\gamma(\lambda u_T)$ is given by
$$
\Omega_\gamma(\lambda u_T) = \{ x\in \R^N \, ; \quad \lambda u_T(x) \leq \gamma \}  = \Omega_{\gamma/\lambda}(u_T),
$$
which satisfies the interior ball condition with radius $r=T$ since $u_T$ is reachable.

Let us now consider $\alpha \in (1,\infty)$. 
For any $u_T\in \mathcal{R}_T$ and $\lambda\in (0,1)$, we shall check that $\lambda u_T$ satisfies the reachability condition from Theorem \ref{thm: reachability condition}.
First of all note that, since $H$ is  of the form \eqref{power-like H},  its Legendre-Fenchel $H^\ast$ transform is given by
$$
H^\ast (q) = \frac{\alpha - 1}{\alpha} |q|^{\frac{\alpha}{\alpha-1}}
$$

Since $u_T\in \mathcal{R}_T$,  for any $x\in \R^N$, there exists $x_0\in \R^N$ and $c\in \R$ such that the function 
$$
\phi (z)  = \frac{\alpha-1}{\alpha} T^{-\frac{1}{\alpha-1}} \left| z-x_0 \right|^{\frac{\alpha}{\alpha-1}} + c
$$
satisfies
$$
\phi (x)  = u_T(x) \qquad \text{and} \qquad \phi(z) \geq u_T(z) \qquad \forall \in \R^N.
$$
Hence, the function $\psi (z) = \lambda \phi(z)$ satisfies
$$
\psi (x)  = \lambda u_T(x) \qquad \text{and} \qquad \psi(z) \geq \lambda u_T(z) \qquad \forall \in \R^N.
$$
Now,  since $\psi(z)$ can be written as
$$
\psi (z) =  \frac{\alpha-1}{\alpha} \left( \dfrac{T}{\lambda^{\alpha-1}}\right)^{-\frac{1}{\alpha-1}} \left| z-x_0 \right|^{\frac{\alpha}{\alpha-1}} + \lambda c,
$$
we deduce,  from Theorem \ref{thm: reachability condition},  that 
$\lambda u_T \in \mathcal{R}_{T'}$ with 
$T' =  \dfrac{T}{\lambda^{\alpha-1}}$. 
Note that $\alpha>1$ and $\lambda\in (0,1)$ imply that $T'>T$.
Finally, since the reachable set is decreasing in time (see Corollary \ref{cor: concave are reachable}), we conclude that $\lambda u_T\in \mathcal{R}_T$.
\end{proof}

\section{Scalar conservation laws}
\label{sec: SCL}

In this section we prove the results given in Theorem \ref{thm: reach charac SCL} and \ref{thm: reachability SCL abs val} concerning the characterization of the reachable set $\mathcal{R}_T^{SCL}$ for the scalar conservation law
\begin{equation}\label{SCL eq}
\left\{\begin{array}{ll}
\partial_t v + \partial_x H(v) = 0 & \text{in} \ (0,T)\times \R \\
v(0, \cdot)= v_0 & \text{in} \ \R.
\end{array}\right.
\end{equation}

\begin{proof}[Proof of Theorem \ref{thm: reach charac SCL}]
The proof consists in checking that condition \eqref{SCL cond} is equivalent to the condition of Theorem \ref{thm: reachability condition} for the function
$$
u_T (x) := 
\displaystyle\int_0^x v_T (y) dy \qquad \forall x \in \R.
$$
Then,  since $\partial_x u_T (x) = v_T(x)$ for a.e. $x\in \R$, we have that $v_T$ is reachable for the equation \eqref{SCL eq} if and only if $u_T$ is reachable for the equation \eqref{HJ intro}.
But, in view of Theorem \ref{thm: reachability condition},
$u_T$ is reachable for \eqref{HJ intro} if and only if, for all $x\in \R$, there exists $x_0$ such that the function 
$$
z \longmapsto\displaystyle\int_0^z v_T (y) dy - T \, H^\ast \left( \dfrac{z-x_0}{T}\right)
$$
has a global maximum at $x$, and the proof is concluded.
\end{proof}

We end the section with the proof of Theorem \ref{thm: reachability SCL abs val} stated in the introduction, which corresponds to the application of Theorem \ref{thm: reach charac SCL} to the case $H(p) = |p|$.

\begin{proof}[Proof of Theorem \ref{thm: reachability SCL abs val}]
First of all, we recall that the Legendre-Fenchel transform of $H(p) = |p|$ is given by the function
\begin{equation}
\label{legendre trans abs val SCL}
H^\ast (q) = \left\{
\begin{array}{ll}
0 & \text{if} \ |q| \leq 1 \\
+\infty & \text{else.}
\end{array}
\right.
\end{equation}

We first prove that \eqref{OSLC abs value} implies \eqref{SCL cond}, and then we will prove the reversed implication.
Let $v_T$ satisfy \eqref{OSLC abs value}.
For any $x\in \R$, define the points
$$
x_1 : = \sup \{ y\in (-\infty, \, x] \ \text{such that} \ v_T (z) \geq 0 \ \text{for a.e.} \ z\in (y,  \, x)  \}
$$
and
$$
x_2 := \inf \{ y\in [x,+\infty) \ \text{such that} \ v_T (z) \leq 0 \ \text{for a.e.}\ z\in (x, \, y) \}.
$$
By the choice of $x_1$ and $x_2$, we have that for any $\varepsilon>0$, the sets 
$$
[x_1-\varepsilon,\, x_1] \cap \{ v_T (y) <0\} \qquad \text{and} \qquad
[x_2,\, x_2+ \varepsilon] \cap \{ v_T(y) >0\}
$$
have both positive measure,
whence, by the assumption \eqref{OSLC abs value},  and letting $\varepsilon\to 0^+$, we deduce that 
\begin{equation}
\label{x_1 x_2 estimate}
x_2 - x_1 \geq 2T.
\end{equation}
Moreover, by the choice of $x_1$ and $x_2$, we have
$v_T(z) \geq 0$ for a.e. $z\in (x_1,x)$ and $v_T(z) \leq 0$ for a.e. $z\in (x,x_2)$.
This implies that the function $g: [x_1,x_2]\to \R$, defined by
\begin{equation}
\label{SCL proof g def}
g(z)  = \int_0^z v_T (y) dy, \qquad \forall z\in [x_1, \, x_2],
\quad \text{has a global maximum at $z = x$.}
\end{equation}
Then, by \eqref{x_1 x_2 estimate}, along with the fact that $x\in (x_1,x_2)$, implies that there exists $x_0\in (x_1,x_2)$ such that
$$
[x_0-T, \, x_0 + T] \subset (x_1,\, x_2) \qquad \text{and} \qquad
x\in [x_0-T, \, x_0 + T]
$$
Finally,  for this choice of $x_0$, and using \eqref{legendre trans abs val SCL}, we obtain that
$$
\int_0^z v_T(y) dy - T\, H^\ast \left( \dfrac{z-x_0}{T}  \right) = 
\left\{
\begin{array}{ll}
g(z) & \text{for} \ z\in [x_0-T, \, x_0 + T] \\
-\infty & \text{else},
\end{array}
\right.
$$
and we can deduce from \eqref{SCL proof g def} that the function 
$$
z\longmapsto \int_0^z v_T(y) dy - T\, H^\ast \left( \dfrac{z-x_0}{T}  \right) 
$$
has a global maximum at $x$.

Let us prove the reversed implication.
Consider a function $v_T\in L^\infty (\R)$ satisfying \eqref{SCL cond}.
For any $x\leq y$, it is obvious that, if $y-x\geq 2T$, then \eqref{OSLC abs value} trivially holds.
It is therefore sufficient to prove that the property \eqref{OSLC abs value} holds in any interval of length $2T$.
Let  $(a,b)\subset \R$ be any interval with $b-a = 2T$, and set
$$
x_1 : =\sup \{ y\in (a,b) \ \text{such that} \ v_T(z) \geq 0 \ \text{for a.e.} \ z\in (a, y)\}.
$$

If $x_1 = b$, then $v_T(x)\geq 0$ for a.e. $x\in (a,b)$, which implies that $v_T(x) >0$ for a.e. $x\in \operatorname{supp}(v_T)\cap (a,b)$, and hence,  property \eqref{OSLC abs value} holds in $(a,b)$.

If we have $x_1\in [a,b)$,  by the definition of $x_1$, it holds that, for any $\varepsilon>0$, there exists 
$x_\varepsilon\in (x_1, x_1+\varepsilon]$ such that
$$
 \int_{x_1}^{x_\varepsilon} v_T(y) dy <0,
$$
which implies that the function 
\begin{equation}\label{g def}
y\longmapsto g(y) = \int_0^y v_T(z) dz
\end{equation}
satisfies $g(x_1) > g(x_\varepsilon)$.
Using the assumption \eqref{SCL cond},  together with the particular form of $H^\ast$ in \eqref{legendre trans abs val SCL}, we have that 
\begin{equation}\label{proof abs val max}
\forall x\in \R,  \quad
\exists x_0 := x_0 (x) \in \R \quad  \text{s.t.}  \quad |x-x_0|\leq T \  \text{and}  \ g(x)\geq g(y) \  \forall y\in [x_0-T,x_0+T].
\end{equation} 

In particular, applying this property to $x_\varepsilon$, and the fact that $b-a=2T$, we have that $g(y)\leq g(x_1)$ for all $y\in [x_1, b]$.

We can now deduce that  $v_T(y) \leq 0$ for a.e.  $y\in (x_1,b)$.
This is indeed equivalent to prove that the function $g(x)$ defined in \eqref{g def}
is nonincreasing in $[x_1, b]$.
Assume for a contradiction that 
$$
\exists z_1, z_2\in [x_0,b] \quad \text{with} \ z_1 <z_2 \ \text{and} \ g(z_2) > g(z_1).
$$
Then we have $g(z_1) < g(z_2) \leq g(x_0)$, which together with $z_1 \in (x_0, z_2) \subset (a,b)$ leads to a contradiction with the statement \eqref{proof abs val max}.
We have then proved that $v_T(x) \geq 0$ for a.e. $x\in (a,x_1)$ and  $v_T(x)\leq 0$ for a.e.  $x\in (x_1,b)$, which implies that \eqref{OSLC abs value} holds in $(a,b)$.

\end{proof}

\section{Conclusions and open questions}
\label{sec: conclusions}

In this work we studied the range of the operator that associates, to any initial condition, the solution at time $T$ of nonlinear first-order partial differential equations such as Hamilton-Jacobi equations and scalar conservation laws.
In the case when the Hamiltonian (resp. the flux) is smooth and uniformly convex, the range of this operator is well-understood, and can be characterized by means of semiconcavity estimates for Hamilton-Jacobi equations, and by one-sided Lipschitz condition for scalar conservation laws.
Our goal in this work was to extend this results to the more general case when the Hamiltonian is not necessarily smooth nor strictly convex, and is merely assumed to be a convex function. Note that in this case, semiconcavity estimates are not available.

Our characterization of the reachable set for Hamilton-Jacobi equations relies on the use of the Hopf-Lax formula for the viscosity solution.
This result is then adapted to the case of scalar conservation laws in one space dimension by using the link between both equations.

In the particular case of Hamilton-Jacobi equations with $H(p) = |p|$, we give a rather geometrical description of the reachable set by means of an interior ball condition on the sublevel sets of the target, which yields a one-sided regularity estimate for the boundary of the sublevel sets.  

Finally, we use our main results to deduce several structural properties of the reachable set.  For instance, we can prove that for power-like Hamiltonians of the form $H(p) = |p|^\alpha$, with $\alpha \geq 1$, the reachable set is star-shaped with center at the origin. Moreover,  if $\alpha = 2$, the reachable set is convex,  and if  $\alpha = 1$, then it consists of a (non-convex) cone.

\textbf{Open questions.} 
Let us conclude the paper with two questions that we were not able to answer, and might be addressed in forthcoming works.

\begin{enumerate}
\item We proved that for the case of Hamilton-Jacobi equations with power-like Hamiltonian, the reachable set is star-shaped, with center at the origin. Although it seems reasonable that the same property should hold for the case of general convex Hamiltonians, we were not able to provide a rigorous proof.
\item Concerning the same star-shaped property for the reachable set, we proved that the origin is a center of the domain, however,  we cannot confirm whether or not other function than zero could be centers of this star-shaped set, i.e. a function $u_T^\ast \in \mathcal{R}_T$ such that
$$
\forall u_T\in \mathcal{R}_T, \ \text{and} \ \forall \lambda\in [0,1], \qquad
 \lambda u_T + (1-\lambda) u_T^\ast \in \mathcal{R}_T.
$$
For instance, the set of concave functions is a convex set contained in the reachable set,  which makes it a good candidate to find other centers.
\end{enumerate}

\bibliographystyle{abbrv}
\bibliography{mybibfile-HJ}

\end{document}